\documentclass[a4paper, 11pt]{amsart}

\usepackage{amsfonts, amsthm, amssymb, amsmath, stmaryrd}
\usepackage{amscd} 
\usepackage{mathrsfs,array}
\usepackage{lmodern}
\usepackage{mathtools}

\usepackage{tabularx} 
\usepackage[all,cmtip,line]{xy}
\usepackage{enumerate}
\usepackage{hyperref}
\usepackage{bm}

\newtheorem{thm}{\textbf{Theorem}}[section]

\newtheorem{prop}[thm]{\textbf{Proposition}}
\newtheorem{lem}[thm]{\textbf{Lemma}}
\newtheorem{corollary}[thm]{\text{Corollary}}
\newtheorem{rem}[thm]{\textbf{Remark}}

\def\A{\mathbb{A}}

\def\R{\mathbb{R}}
\def\Z{\mathbb{Z}}

\def\A{\mathbb{A}}
\def\cN{\mathcal{N}}

\newcommand{\cL}{\mathcal{L}}

\def\cS{\mathcal{S}}

\def\inf{\mathrm{inf}}

\def\pp{\mathfrak{p}}
\def\m{\mathfrak{m}}
\def\a{\mathfrak{a}}

\newcommand{\lb}{\llbracket}
\newcommand{\rb}{\rrbracket}

\usepackage{color}

\begin{document}

\title{$\A_{\inf}$ is infinite dimensional}

\author{Jaclyn Lang}
\address{LAGA, UMR 7539, CNRS, Universit\'e Paris 13 - Sorbonne Paris Cit\'e, Universit\'e Paris~8}
\email{lang@math.univ-paris13.fr}

\author{Judith Ludwig}
\address{IWR, University of Heidelberg, Im Neuenheimer Feld 205, 69120 Heidelberg, Germany}
\email{judith.ludwig@iwr.uni-heidelberg.de}

\subjclass[2010]{}

\maketitle
\begin{abstract} Given a perfect valuation ring $R$ of characteristic $p$ that is complete with respect to a rank-$1$ nondiscrete valuation, we show that the ring $\A_{\inf}$ of Witt vectors of $R$ has infinite Krull dimension.
\end{abstract}
\section{Introduction}
Fix a prime $p$.  Let $R$ be a perfect valuation ring of characteristic $p$ and denote the valuation by $v$. Assume $v$ is of rank $1$ and nondiscrete and that $R$ is complete with respect to $v$. Let $\A: =\A_{\inf} \coloneqq W(R)$ be the ring of Witt vectors of $R$.  This ring plays a central role in $p$-adic Hodge theory as it is the basic ring from which all of Fontaine's $p$-adic period rings are built. It is also central to the construction of the (adic) Fargues--Fontaine curve \cite{FarguesFontaine}.  Recently, Bhatt, Morrow and Scholze constructed $\A_{\inf}$-cohomology, a cohomology theory that specializes to \'etale, de Rham and crystalline cohomology \cite{BMS}. In these works there is a useful analogy between $\A$ and a 2-dimensional regular local ring.  In this paper we prove the following theorem.

\begin{thm}\label{main thm}
The ring $\A$ has infinite Krull dimension.
\end{thm}
Bhatt \cite[Warning 2.24]{Bhatt} and Kedlaya \cite[Remark 1.6]{Kedlaya} note that the Krull dimension of $\A$ is at least 3.  To see this, fix a pseudouniformizer $\varpi \in R$ and let $\kappa$ denote the residue field of $R$.  Let $W(\m)$ be the kernel of the natural map $W(R) \to W(\kappa)$ and $[-] \colon R \to W(R)$ the Teichm\"uller map. Then Bhatt and Kedlaya point out that $\A$ contains the following explicit chain of prime ideals 
\[
(0) \subset \pp \coloneqq \bigcup_{k = 0}^\infty [\varpi^{1/p^k}]\A \subset W(\mathfrak{m}) \subset (p,W(\mathfrak{m})).
\]
As suggested in \cite[Remark 1.6]{Kedlaya}, we use Newton polygons to find an infinite chain of prime ideals between $\pp$ and $W(\mathfrak{m})$. 

The equal characteristic analogue of Theorem \ref{main thm} is the statement that the power series ring $R\lb X\rb$ has infinite Krull dimension.  This was first proved by Arnold \cite[Theorem 1]{Arnold}, and the structure of our argument is very similar to his.  We axiomatize Arnold's argument in Section \ref{strategy section}.  

\textbf{Notation.} We use the convention that the symbols $<, >, \subset, \supset$ denote strict inequalities and inclusions with the exception that we allow the statement ``$\infty < \infty$" to be true.  Otherwise, if equality is allowed it will be explicitly reflected in the notation using the symbols $\leq, \geq, \subseteq, \supseteq$.  An inequality between two $(\R \cup \{\pm \infty\})$-valued functions means that the inequality holds point-wise.

\textbf{Acknowledgements.}  
We thank Kevin Buzzard for helpful comments on an earlier draft.  The first author gratefully acknowledges support from the National Science Foundation through award DMS-1604148.  

\section{Review of Newton polygons}
As above let $R$ be a perfect valuation ring of characteristic $p$ that is complete with respect to a nondiscrete valuation $v$ of rank 1. Let $\mathfrak{m}$ be the maximal ideal of $R$, and fix an element $\varpi \in \mathfrak{m}$ of valuation $v(\varpi)=1$.

Let $\A:= W(R)$ be the ring of Witt vectors of $R$.  Write $[-] \colon R \to \A$ for the Teichm\"uller map, which is multiplicative.  Recall that every element of $\A$ can be written uniquely in the form $\sum_{n \geq 0} [x_n]p^n$ with $x_n \in R$.

As in \cite[Section 1.5.2]{FarguesFontaine}, given $f \in \A$ with $f = \sum_{n\geq 0} [x_n] p^n$, we define the \textit{Newton polygon} $\cN(f)$ of $f$ as the largest decreasing convex polygon in $\R^2$ lying below the set of points $\{(n,v(x_n)) \colon n \geq 0\}$.  We shall often view $\cN(f)$ as the graph of a function $\cN(f) \colon \R \to \R \cup \{+\infty\}$.  In particular, if $n_f$ is the smallest integer such that $x_{n_f} \neq 0$, then $\cN(f)(t)=+ \infty$ for $t< n_f$ and $\cN(f)(n_f) = v(x_{n_f})$. Furthermore, $\lim_{t\rightarrow \infty} \cN(f)(t) = \inf_n v(x_n)$.  

Following the conventions in \cite[Section 1.5.2]{FarguesFontaine}, for any integer $i \geq 0$ define 
\[s_i(f) \coloneqq \cN(f)(i) - \cN(f)(i+1).\]
We call $s_i(f)$ the \textit{slope} of $\cN(f)$ on the interval $[i, i+1]$ even though one would typically call that slope $-s_i(f)$.  With this convention, the slopes form a nonnegative decreasing sequence; that is, $s_i(f) \geq s_{i+1}(f) \geq 0$ for all $i$.  We say that $n$ is a \textit{node} of $\cN(f)$ if $\cN(f)(n) = v(x_n)$.  

We recall the theory of Legendre transforms from \cite[Section 1.5.1]{FarguesFontaine}.  Given a function $\varphi \colon \R \to \R \cup \{+\infty\}$ that is not identically equal to $+\infty$, define
\begin{align*}
\cL(\varphi) \colon \R &\to \R \cup \{-\infty\}\\
\lambda &\mapsto \inf_{t \in \R} \{\varphi(t) + \lambda t\}.
\end{align*}
If $\varphi$ is a convex function then one can recover $\varphi$ from $\cL(\varphi)$ via the formula
\[
\varphi(t) = \sup_{\lambda \in \R} \{\cL(\varphi)(\lambda) - t\lambda\}.
\]
From these definitions it is easy to see that $\cN(f) \leq \cN(g)$ if and only if $\cL(\cN(f)) \leq \cL(\cN(g))$.    

As explained in \cite[Section 1.5]{FarguesFontaine}, for any $f, g \in \A$ we have 
\begin{equation}\label{adding Legendre transforms}
\cL(\cN(fg)) = \cL(\cN(f)) + \cL(\cN(g)).
\end{equation}
Motivated by this, one defines a convolution product on the set of $(\R \cup \{+\infty\})$-valued convex function on $\R$ that are not identically $+\infty$ by
\[
(\varphi*\psi)(t) \coloneqq \sup_{\lambda \in \R}\{\cL(\varphi)(\lambda) + \cL(\psi)(\lambda) - t\lambda\}.
\]
Thus we have $\cN(fg) = \cN(f)*\cN(g)$.  In particular if $\cN(f) > 0$, then $\cN(f^m) < \cN(f^{m+1})$ for all $m \geq 1$, and for any $t \in \R$ we have $\lim_{m \to \infty} \cN(f^m)(t) = +\infty$.  

There is another way of describing $\cN(fg)$ in terms of $\cN(f)$ and $\cN(g)$ without explicitly using Legendre transforms.  Write $f = \sum_{n \geq 0}[x_n]p^n$ and $g = \sum_{n \geq 0} [y_n]p^n$, and let $n_f$ (respectively, $n_g$) be the smallest integer such that $x_n \neq 0$ (respectively, $y_n \neq 0$).  Then $\cN(fg)(t) = +\infty$ for all $t < n_f+n_g$, and $\cN(fg)(n_f+n_g) = v(x_{n_f}) + v(y_{n_g})$.  The slopes of $\cN(fg)$ are given by interlacing the slopes of $\cN(f)$ and $\cN(g)$.  That is, the slope sequence of $\cN(fg)$ is given by combining the sequences $\{s_i(f) \colon i \geq 0\}$ and $\{s_i(g) \colon i \geq 0\}$ into a single decreasing sequence that incorporates all positive elements of both sequences.  
The relationship between this description and equation \eqref{adding Legendre transforms} is explained in \cite[Section 1.5]{FarguesFontaine}.

\begin{lem}\label{inequalities far out}
Let $f$ be an element of $\A$ such that $\cN(f) > 0$.  If $g \in \A$ and $t_0 \geq 0$ such that for all $t \geq t_0$ we have $\cN(g)(t) \leq \cN(f)(t)$, then for all $m$ sufficiently large we have $\cN(g) \leq \cN(f^m)$.
\end{lem}

\begin{proof}
As noted above, since $\cN(f) > 0$, the sequence $\{\cN(f^m)\}_m$ converges to $+\infty$.  This convergence is uniform on the compact interval $[0, t_0]$.  Thus for $m$ sufficiently large, it follows that $\cN(g)(t) \leq \cN(f^m)(t)$ for all $t \in [0, t_0]$.  On the other hand, for all $t \geq t_0$ we have
\[
\cN(g)(t) \leq \cN(f)(t) < \cN(f^m)(t).
\]
Thus $\cN(g) \leq \cN(f^m)$ for all $m$ sufficiently large.
\end{proof}

\begin{prop}\label{p is prime}
The ideal $\pp \coloneqq \bigcup_{k = 0}^\infty [\varpi^{1/p^k}]\A$ is a prime ideal of $\A$. 
\end{prop}
\begin{proof}
Note that an element $f$ of $\A$ lies in $\pp$ if and only if $\lim_{t\rightarrow \infty} \cN(f)(t) >0$. If $g,g' \in \A\backslash \pp$, then $\lim_{t\rightarrow \infty} \cN(gg')(t)= \lim_{t\rightarrow \infty} (\cN(g)*\cN(g'))(t) = 0$ and so $gg' \notin \pp$. 
\end{proof}

\section{The strategy}\label{strategy section}
We define infinitely many sequences in $R$ as follows.  For all $i \geq 0$, define $a_{1, i} \coloneqq \varpi^{1/p^i} \in R$.  For $n > 1$ and $i \geq 0$, define $a_{n, i}$ recursively by
\[
a_{n,i} \coloneqq a_{n-1, i^2} \in R.
\]
Thus $a_{n,i}= \varpi^{1/p^{n_i}}$, where $n_i \coloneqq i^{2^{n-1}}$, and $v(a_{n,i})= p^{-n_i}$.
For each $n \geq 1$, define
\[
h_n \coloneqq \sum_{i = 0}^\infty [a_{n, i}]p^i \in \A.
\]
Note that $\cN(h_n) > 0$, for any $n$ we have $\lim_{t \to \infty} \cN(h_n)(t) = 0$, and $\cN(h_n)$ has a node at every integer.  

Finally, we define the following subsets of $\A$.  For $n \geq 1$, let
\[
\cS_n \coloneqq \{g \in \A \colon 0 < \cN(g) \leq \cN(h_n^m) \text{ for some } m \geq 1 \}.
\]
In particular, $h_n \in \cS_n$.

\begin{prop}\label{strategy prop}
The sets $\cS_n$ satisfy the following three properties:
\begin{enumerate}
\item\label{hyp1} for all $n \geq 1$ we have $\cS_{n + 1} \subset \cS_n$;
\item\label{hyp2} each $\cS_n$ is multiplicatively closed;
\item\label{hyp3} for any $g \in \cS_{n + 1}$ and $f \in \A$, we have that $g + fh_n \in \cS_{n+1}$.
\end{enumerate}
\end{prop}

We prove this proposition in Section \ref{Sec:4}.   

\begin{thm} The ring $\A$ has infinite Krull dimension.
\end{thm}

\begin{proof}
We follow Arnold's proof of \cite[Theorem 1]{Arnold}.  We prove that for any $n \geq 1$, there exists a chain of prime ideals of $\A$, say $\pp_1 \subset \cdots \subset \pp_n$, such that $\pp_n \cap \cS_n = \emptyset$.

For $n = 1$, let $\pp_1 = \pp$.  To see that $\pp \cap \cS_1 = \emptyset$, note that if $f \in \pp$, then $f \in [\varpi^{1/p^k}]\A$ for some $k \geq 0$, and so $\cN(f) \geq 1/p^{k}$.  On the other hand, if $f \in \cS_1$, then for some $m \geq 1$ we have that $\lim_{t \to \infty} \cN(f)(t) \leq \lim_{t \to \infty} \cN(h_1^m)(t) = 0$.

Fix $n \geq 1$ and suppose for induction that there is a chain $\pp_1 \subset \cdots \subset \pp_n$ of prime ideals of $\A$ such that $\pp_n \cap \cS_n = \emptyset$.  Consider the ideal $\a_n \coloneqq \pp_n + h_n\A$.  Note that $\a_n \neq \pp_n$ since $h_n \in \cS_n$ and $\pp_n \cap \cS_n = \emptyset$.  We claim that $\a_n \cap \cS_{n+1} = \emptyset$.  Indeed, given $g \in \cS_{n+1}$, we have that $g + h_nf \in \cS_{n+1}$ for all $f \in \A$ by property \eqref{hyp3} of the sets $\cS_n$.  By property \eqref{hyp1}, it follows that $g + h_nf \in \cS_n$ for all $f \in \A$.  If $g \in \a_n$, then there is some $f \in \A$ such that $g + h_nf \in \pp_n$. But $\pp_n \cap \cS_n = \emptyset$, so it follows that $g \not\in \a_n$.

Since $\cS_{n+1}$ is multiplicatively closed by property \eqref{hyp2}, there is a prime ideal $\pp_{n+1}$ of~$\A$ such that $\pp_n \subset \a_n \subseteq \pp_{n+1}$ and $\pp_{n+1} \cap \cS_{n+1} = \emptyset$.  By induction on $n$, it follows that $\A$ has infinite Krull dimension.  
\end{proof}

\begin{rem}
\begin{enumerate}[a)]
\item[]
\item Arnold has used an argument as above to show that the ring $R \lb X \rb$ has infinite Krull dimension \cite[Theorem 1]{Arnold}.  In fact given any ring $A$, if one can exhibit elements $h_n$ of $A$ and sets $\cS_n$ satisfying the properties in Proposition \ref{strategy prop} together with a prime ideal $\pp$ such that $\pp \cap \cS_1 = \emptyset$, then the above argument shows that $A$ has infinite Krull dimension.
\item There is a rigorous way to view the power series ring $R \lb X \rb$ as an equal characteristic version of $\A$ (see \cite[Section 1.3]{FarguesFontaine}).  Our definitions make sense in this more general setting, and our arguments give another proof that $R\lb X \rb$ has infinite Krull dimension.
\end{enumerate}
\end{rem}

\section{The proof of Proposition \ref{strategy prop}}\label{Sec:4}
In this section we prove Proposition \ref{strategy prop}.  Recall that $v$ is the valuation on $R$ and $s_i(h_n) \coloneqq v(a_{n, i-1}/a_{n,i})$ is the $i$-th slope of $\cN(h_n)$.

\begin{prop}\label{check hyp1}
Fix $n, m \geq 1$.  For $t > 2m^2$ we have that
\[\cN(h_{n+1}^m)(t) < \cN(h_n)(t).
\]
\end{prop}

\begin{proof}
Let $\ell = km + r \in \Z$ with $k > 2m$ and $0 \leq r < m$.  We have
\[
\cN(h_n)(\ell) = v(a_{n, \ell}) = v(a_{n, km+r})
\]
and 
\[
\cN(h_{n+1}^m)(\ell) = mv(a_{n+1, k}) - s_{k+1}(h_{n+1})r < mv(a_{n+1, k}) = mv(a_{n, k^2}).
\]
To see that $mv(a_{k^2}) < v(a_{n, km+r})$, recall that $v(a_{n, i}) = p^{-i^{2^{n-1}}}$.  Thus we must show that
\[
m < p^{k^{2^n} - (km+r)^{2^{n-1}}}.
\]
Since $r < m$, it suffices to show that $m < p^{k^{2^n} - ((k+1)m)^{2^{n-1}}}$.  One checks this quickly using that $k > 2m$ and therefore $k^2 - (km +m)>m$.
\end{proof}

\begin{corollary}\label{check hyp1 corollary}
For all $n \geq 1$ we have $\cS_{n+1} \subset \cS_n$.
\end{corollary}

\begin{proof}
If $g \in \cS_{n+1}$ then for some $m \geq 1$ we have $0 < \cN(g) \leq \cN(h_{n+1}^m)$.  By Proposition \ref{check hyp1} and Lemma \ref{inequalities far out} it follows that for $m'$ sufficiently large (depending on $m$ and $n$) we have $\cN(h_{n+1}^m) < \cN(h_n^{m'})$, so $g \in \cS_n$.  To see that the inclusion is strict, note that Proposition \ref{check hyp1} also implies that $h_n \not\in \cS_{n+1}$, but $h_n \in \cS_n$.
\end{proof}

\begin{prop}\label{check hyp2}
For each $n \geq 1$, the set $\cS_n$ is multiplicatively closed.
\end{prop}

\begin{proof}
Let $f, g \in \cS_n$.  Then it is clear that $\cN(fg)= \cN(f)*\cN(g)>0$. 

For $m$ sufficiently large, we have $0 < \cN(f), \cN(g) \leq \cN(h_n^m)$.  Thus for any $\lambda, t \in \R$ we have 
\[
\cN(f)(t) + \lambda t \leq \cN(h_n^m)(t) + \lambda t.
\]
Taking the infimum over $t \in \R$, it follows that $\cL(\cN(f))(\lambda) \leq \cL(\cN(h_n^m))(\lambda)$ for all $\lambda \in \R$.  Similarly, $\cL(\cN(g)) \leq \cL(\cN(h_n^m))$.  Therefore
\[
\cL(\cN(fg)) = \cL(\cN(f)) + \cL(\cN(g)) \leq 2\cL(\cN(h_n^m)) = \cL(\cN(h_n^{2m})).
\]
Hence we have that $\cL(\cN(fg))(\lambda) - t\lambda \leq \cL(\cN(h_n^{2m}))(\lambda) - t\lambda$ for all $t, \lambda \in \R$.  It follows that
\[
\cN(fg)(t) = \sup_\lambda\{\cL(\cN(fg))(\lambda) - t\lambda\} \leq \sup_\lambda\{\cL(\cN(h_n^{2m}))(\lambda) - t\lambda\} = \cN(h_n^{2m})(t)
\]
for all $t \in \R$.  Therefore $fg \in \cS_n$.  
\end{proof}

\begin{prop}\label{multiplication raises np}
Let $h$ be an element of $\A$ such that $\cN(h) >0$. Then for any $f \in \A$, $\cN(fh) \geq \cN(h)$. 
\end{prop}

\begin{proof}
The Newton polygon $\cN(fh)$ starts at $n_f+n_h$. Note that the slopes of $\cN(fh)$ are all positive and form a monotone sequence converging to zero.  Therefore all slopes $s_i(h)$ of $h$ eventually occur as slopes of $\cN(hf)$. It follows that for any $l\geq n_f+n_h$, $\cN(fh)(l)\geq \sum_{i\geq l}^{\infty} s_i(h)= \cN(h)(l)$. 
\end{proof}

Let $f, g \in \A$, and write $f = \sum_{n = 0}^\infty [x_n]p^n$ and $g = \sum_{n = 0}^\infty[y_n]p^n$.  In order to prove property (3) from Proposition \ref{strategy prop} we need to understand the Newton polygon of $f+g$ in terms of those of $f$ and $g$. For that we show a property of Witt vector addition in Lemma \ref{witt vector addition lemma} below. First, recall the translation between Teichm\"uller expansions and Witt coordinates:
\[
\sum_{n = 0}^\infty [x_n]p^n = (x_0, x_1^p, x_2^{p^2}, \ldots, x_n^{p^n}, \ldots).
\]
Recall also that addition of Witt vectors is governed by the polynomials 
\[
S_n(X_0, \ldots, X_n; Y_0, \ldots, Y_n), 
\]
which are defined recursively by
\[
S_0(X_0; Y_0) \coloneqq X_0 + Y_0
\]
and
\[
\sum_{k = 0}^n p^kS_k(X_0, \ldots, X_k; Y_0, \ldots, Y_k)^{p^{n-k}} = \sum_{k = 0}^n p^k(X_k^{p^{n - k}} + Y_k^{p^{n-k}}).
\]
Thus
\begin{align*}
f + g &= (S_0(x_0; y_0), \ldots, S_n(x_0, \ldots, x_n^{p^n}; y_0, \ldots, y_n^{p^n}), \ldots)\\
&= \sum_{n = 0}^\infty [S_n(x_0, \ldots, x_n^{p^n}; y_0, \ldots, y_n^{p^n})^{p^{-n}}]p^n.
\end{align*}

\begin{lem}\label{witt vector addition lemma}
For all $n \geq 0$ we have that 
\[
S_n(x_0, \ldots, x_n^{p^n}; y_0, \ldots, y_n^{p^n}) = x_n^{p^n} + y_n^{p^n} + \Sigma_n,
\]
where $\Sigma_n$ is a sum of terms of the form $\prod_{k = 0}^{n-1} x_k^{p^ki_k}y_k^{p^kj_k}$ such that $\sum_{k = 0}^{n - 1} p^k(i_k + j_k) = p^n$.
\end{lem}

\begin{proof}
Note that if the lemma holds for some $n$, then $S_n^p$ is a sum of terms of the form $\prod_{k = 0}^n x_k^{p^ki_k}y_k^{p^kj_k}$ such that $\sum_{k = 0}^n p^k(i_k + j_k) = p^{n+1}$.  The lemma then follows from the definition of $S_n$ and induction on $n$.
\end{proof}

\begin{prop}\label{Newt polys of sums}
Let $f= \sum_{n = 0}^\infty [x_n]p^n, g = \sum_{n = 0}^\infty[y_n]p^n \in \A$.  Assume that $\cN(g)$ is strictly decreasing.  Suppose there exists a $t_0 \geq 0$ such that for all $t \geq t_0$ we have $\cN(g)(t) < \cN(f)(t)$.  If $n$ is a node of $\cN(g)$ such that $n \geq t_0$, then $\cN(g + f)(n) \leq \cN(g)(n)$.  
\end{prop}

\begin{proof}
Since $\cN(g)$ is strictly decreasing and $n$ is a node of $\cN(g)$, we have that
\[
v(y_n) = \cN(g)(n) < v(y_m)
\]
for all $m < n$.  Since $n \geq t_0$ and $\cN(f)$ is decreasing, for all $m \leq n$ we have that
\[
v(y_n) = \cN(g)(n) < \cN(f)(n) \leq v(x_m).
\]
Thus $v(y_n^{p^n}) < v(x_n^{p^n})$ and for any $i_0, j_0, \ldots, i_{n-1}, j_{n-1}$ such that $\sum_{k = 0}^{n-1} p^k(i_k + j_k) = p^n$, it follows that
\[
v\bigl(\prod_{k = 0}^{n-1} y_k^{p^ki_k}x_k^{p^kj_k} \bigr) > p^nv(y_n) = v(y_n^{p^n}).
\]
By Lemma \ref{witt vector addition lemma} it follows that 
\[
v(S_n(y_0, \ldots, y_n^{p^n}; x_0, \ldots, x_n^{p^n})^{p^{-n}}) = v(y_n).
\]
Therefore 
\[
\cN(g+f)(n) \leq v(S_n(y_0, \ldots, y_n^{p^n}; x_0, \ldots, x_n^{p^n})^{p^{-n}}) = v(y_n) = \cN(g)(n).
\]
\end{proof}

\begin{corollary}\label{check hyp3}
If $g \in \mathcal{S}_{n+1}$ and $f \in \A$, then $g+fh_n \in \cS_{n+1}$.
\end{corollary}

\begin{proof}
Since $g \in \cS_{n+1}$, it follows that $\cN(g)$ is strictly decreasing and there exists $m\geq 0$ such that  $\cN(g) \leq \cN(h^m_{n+1})$.  By Proposition \ref{check hyp1} and Proposition \ref{multiplication raises np}, for all $t > 2m^2$ we have 
\[
\cN(g)(t) \leq \cN(h^m_{n+1})(t) <\cN(h_{n})(t)\leq \cN(fh_n)(t).
\]
Note that $\cN(fh_n)$ is strictly decreasing since all slopes of $\cN(h_n)$ occur as slopes of $\cN(fh_n)$.  By Proposition \ref{Newt polys of sums} it follows that for all $t$ sufficiently large, 
\[
\cN(g + fh_n)(t) \leq \cN(g)(t) \leq \cN(h_{n+1}^m)(t).  
\]
By Lemma \ref{inequalities far out} it follows that $g + fh_n \in \cS_{n+1}$.
\end{proof}

\bibliography{AinfReferences}
\bibliographystyle{alpha}

\end{document}